\documentclass[10pt]{amsart}

\usepackage{amsfonts}

\usepackage{graphicx}
\usepackage{amsmath}
\usepackage{setspace}
\usepackage[tmargin=2.0cm, bmargin=2.0cm, lmargin=3.0cm, rmargin=3.0cm]{geometry}
\usepackage{url,ulem}
\usepackage[bookmarks=false,unicode,colorlinks,urlcolor=red,citecolor=red,linkcolor=blue]{hyperref}
\newcommand{\halmos}{{\mbox{\, \vspace{3mm}}} \hfill
\mbox{$\Box$}}

\usepackage{amsmath,amssymb}\usepackage{amsthm}
\usepackage{enumerate}
\usepackage[T1]{fontenc}
\usepackage{ifxetex}
\usepackage{bbm}
\usepackage{ulem}

\usepackage{color}
\definecolor{lw}{RGB}{255,0,0}
\definecolor{kb}{RGB}{0,255,0}

\newtheorem{thm}{Theorem}
\newtheorem{prop}[thm]{Proposition}
\newtheorem{lemm}[thm]{Lemma}
\newtheorem{cor}[thm]{Corollary}

\theoremstyle{remark}

\newtheorem{example}[thm]{Example}

\allowdisplaybreaks[4]

\def\R{{\mathbb{R}}}
\def\E{{\mathbb{E}}}
\def\P{{\mathbb{P}}}

\def\im{{\mathrm{i}}}

\def\S{{\mathbb{S}}}

\def\Rd{{{\R}^d}}

\def\sA{{\mathcal{A}}}

\def\e{{\mathrm{e}}}
\def\rd{{\mathrm{d}}}

\def\id{{\mathbf{1}}}
\def\CS{{\mathrm{CS}}}

\DeclareMathOperator*\dist{dist}

\makeatletter
\newcommand*{\rmnum}[1]{\expandafter\@slowromancap\romannumeral #1@}
\makeatother

\title{Yaglom limit for stable processes in cones}

\author{Krzysztof Bogdan}
\address{ Faculty of Pure and Applied Mathematics,
Wroc\l aw University of Science and Technology,
Wyb. Wyspia\'nskiego 27, 50-370 Wroc\l aw, Poland}
\email{bogdan@pwr.wroc.pl}

\author{Zbigniew Palmowski}
\address{ Faculty of Pure and Applied Mathematics,
Wroc\l aw University of Science and Technology,
Wyb. Wyspia\'nskiego 27, 50-370 Wroc\l aw, Poland}
\email{zbigniew.palmowski@gmail.com}

\author{Longmin Wang}
\address{School of Mathematical Sciences, Nankai University, Tianjin 300071, P.R. China}
\email{wanglm@nankai.edu.cn}

\thanks{
KB was partially supported by National Science Centre of Poland (NCN)
under grant 2014/14/M/ST1/00600.
ZP
was partially supported
by the
National Science Centre of Poland (NCN)
under grant 2015/17/B/ST1/01102.
LW was partially supported by NSFC under grant 11671216. ZP and LW
also acknowledge partial support by the project RARE -318984, a Marie Curie IRSES Fellowship within the 7th European Community Framework Programme.  }

\date{\today}
\subjclass[2000]{31B05, 60J45, 60F05, 60G51}
\keywords{}

\begin{document}

\begin{abstract}
We  give the asymptotics of the
tail of the distribution of the first exit time of the isotropic $\alpha$-stable L\'evy process
from
the  Lipschitz cone in $\mathbb{R}^d$.
We obtain the Yaglom limit for the killed stable process in the cone.  We construct and estimate
entrance laws for the  process from the vertex into the cone.
For the symmetric Cauchy process
and the positive half-line we give a spectral representation of the Yaglom limit.

Our approach relies on the scalings of the stable process and the
cone, which allow us to express the temporal asymptotics of
the distribution of the process
at infinity by means of
 the spatial asymptotics of harmonic functions of the process at the vertex; on
 the representation of the probability of survival of the process in the cone as a Green potential; and on the approximate factorization of the 
heat kernel of the cone, which secures compactness and yields a limiting (Yaglom) measure by means of Prokhorov's theorem.

\vspace{3mm}

\noindent {\sc Keywords.}  Yaglom limit $\star$ stable process $\star$ Lipschitz cone $\star$ quasi-stationary measure $\star$
Green function $\star$ Martin kernel $\star$ excursions.

\end{abstract}

\maketitle

\pagestyle{myheadings} \markboth{\sc K.\ Bogdan --- Z.\ Palmowski
--- L.\ Wang} {\sc Yaglom limit for stable processes in cones}

\vspace{1.8cm}

\tableofcontents

\newpage

\section{Introduction}\label{sec:intro}

Let $0<\alpha<2$, $d=1,2,\ldots$, and let $X=\{X_t,t\geq 0\}$ be the isotropic $\alpha$-stable L\'evy process in $\Rd$.
We denote by $\mathbb{P}_x$
and $\mathbb{E}_x$ the probability and expectation for the process  starting from 
any $x\in\Rd$, see Section~\ref{sec:prel} for details.
Let $\Gamma\subset\mathbb{R}^d$ be an arbitrary Lipschitz cone with vertex at the origin $0$. We define
\begin{equation}\label{deftau}
\tau_\Gamma=\inf\{t>0: X_t \notin\Gamma\}\,,
\end{equation}
the  time of the first exit of $X$ from $\Gamma$.
The following measure $\mu$ will be called the Yaglom limit for $X$ and $\Gamma$.
\begin{thm}\label{thm:eYl}
There is a probability measure $\mu$ concentrated on $\Gamma$ such that for every Borel set $A\subset \Rd$,
\begin{equation}\label{def1}\lim_{t
\to\infty}\mathbb{P}_x\left(\frac{X_t}{t^{1/\alpha}}\in A \big| \tau_\Gamma >t\right)=\mu (A)\,,\quad x\in \Gamma\,.
\end{equation}
\end{thm}
The above condition $\tau_\Gamma>t$ means that $X$ stays, or survives,
in $\Gamma$ for time longer than $t$. Theorem \ref{thm:eYl} asserts
that, 
given its survival, 
$X_t$ rescaled by $t^{1/\alpha}$ has  a limiting distribution independent of the starting point. We note that rescaling
is essential for the limit to be nontrivial.
The
Yaglom limit $\mu$ corresponds with the idea of 'quasi-stationarity', as
expressed by Bartlett \cite{MR0118550}:
\begin{quote}
It still may happen that the
time to extinction is so long
that it is still of more relevance to consider the effectively
ultimate distribution (called a quasi-stationary distribution) [...]
\end{quote}
Namely, $\mu$
is a quasi-stationary distribution for
$(t+1)^{-1/\alpha} X_t$ in the following
sense.
\begin{prop}
  \label{prop:qs}
Let $\mathbb{P}_\mu(\cdot)=\int_{\Gamma}\mathbb{P}_y(\cdot)\,\mu(\rd y)$.
For every Borel set $A\subset\Rd$,
\begin{equation}\label{eq:to be proved}
\mathbb{P}_\mu\left(\frac{X_t}{(t+1)^{1/\alpha}}\in A\big|\tau_\Gamma >t\right) = \mu(A)\,,\quad t\ge 0\,.%, \quad A\in\mathcal{B}( \Rd).
\end{equation}
\end{prop}
Note that $Y_t=(t+1)^{-1/\alpha} X_t$ is a time-inhomogenous Markov process and under $\mathbb{P}_\mu$, the law of $Y_0$ is $\mu$. 

\bigskip

This is the first paper where the
Yaglom limit  is identified for the multi-dimensional
$\alpha$-stable
L\'evy processe.
For the one-dimensional self-similar processes, including the  symmetric $\alpha$-stable  L\'evy process  in the one-dimensional cone $\Gamma=(0,\infty)$, Yaglom limits similar to \eqref{def1},
and also using rescaling, were given by Haas and Rivero \cite{MR2971725}.
Their proofs
rely on precise estimates for the tail distribution of exponential functionals of non-increasing L\'evy processes and are completely different from ours.
As we will see below, the Yaglom limit may be obtained from the asymptotics (i.e. limits) of the  
survival probability $\mathbb{P}_x(\tau_\Gamma>t)$.
We note that such
asymptotics 
were studied for the multi-dimensional Brownian motion by DeBlassie \cite{MR863716}. Ba\~{n}uelos and Smits \cite{MR1465162} gave the asymptotics of the heat kernel of the cone
in terms of the orthonormal eigenfunctions of the Laplace-Beltrami operator on the cone's spherical cap.
Denisov and Wachtel \cite{MR3342657} derived a result similar to Theorem~\ref{thm:eYl} for multidimensional random walks
by using coupling with the Brownian motion.
The tail distribution of $\tau_\Gamma$ for the isotropic  $\alpha$-stable L\'evy process  and wedges $\Gamma\subset \mathbb{R}^2$ was estimated by DeBlassie \cite{MR1062058}. Ba\~{n}uelos and Bogdan \cite{MR2075671} also
provided estimates but not asymptotics for
general cones in $\Rd$. They used
the boundary Harnack principle (BHP),
which turns out to be very useful also in our situation,
because
it, in fact, yields 
the 
asymptotics of the survival probability $\mathbb{P}_x(\tau_\Gamma>1)$
for $x\to 0$, as we show below.
These asymptotics %is
are given
in Theorem~\ref{T:survival}, and
they lead to
Theorem~\ref{thm:eYl}, to sharp estimates for the density function of the Yaglom limit $\mu$,
and to
the
existence and estimates of laws of
excursions of  the stable process from the vertex into the cone, which we give in Theorem~\ref{mainresult}.

Information on quasi-stationary (QS) distributions for time-homogeneous Markov processes can be found in  the classical works of Seneta and Vere-Jones \cite{MR0207047}, Tweedie \cite{MR0443090}, Jacka and Roberts \cite{MR1363332}. The bibliographic database of Pollet \cite{pollett} gives detailed history of
QS distributions. In particular, 
Yaglom \cite{MR0022045} was the first to explicitly  identify QS
distributions 
for the subcritical Bienaym\'e-Galton-Watson branching process.
Part of the results on QS distributions
concern Markov chains on positive integers with an absorbing state at the origin
\cite{MR2986807, MR1133722, MR1334159, MR0346932, MR0207047,
  MR3247530}. 
Other objects of study are the extinction probabilities for continuous-time branching process and the Fleming-Viot process \cite{MR3498004, MR2318407, MR2299923}.
A separate topic %are
is the one-dimensional L\'evy processes exiting from the positive half-line.
Here the case of the Brownian motion with drift was resolved by Martinez and San Martin \cite{MR1303922}, complementing the result for random walks obtained by Iglehart \cite{MR0368168}. The case of jump L\'evy processes was studied by   E.~Kyprianou \cite{MR0292201}, A.~Kyprianou and Palmowski \cite{MR2299923} and Mandjes et al. \cite{MR2959448}.
These papers
are based on the Wiener-Hopf factorization and
Tauberian theorems. They are intrinsically one-dimensional and they do not use the boundary asymptotics of harmonic functions or rescaling to obtain the limiting distribution.
We also note in passing that these results relate to
the behavior of the one-dimensional L\'evy processes and random walks conditioned to stay positive, for which we refer the reader to Bertoin \cite{MR1232850}, Bertoin and Doney \cite{MR1372336}, and Chaumont and Doney \cite{MR2164035}.

On a general level our development depends on
a compactness argument based on recent sharp estimates of the heat kernel of cones and on a formula expressing the survival probability $\mathbb{P}_x(\tau_\Gamma>t)$ as a Green potential.
The latter allows us to obtain the spatial asymptotics of the survival probability at
the vertex of the cone $\Gamma$ in terms of the cone's Martin kernel with the pole at infinity, and is a consequence of BHP.
By scaling we then obtain the asymptotics of the survival probability
as $t\rightarrow\infty$.
The construction allows for the identification of the limiting
boundary behavior of the heat kernel at the vertex of the cone. Such
asymptotics
are completely new, and may be regarded as a culmination of the study of the Dirichlet fractional Laplacian, which started with boundary estimates and asymptotics of harmonic functions, developed into estimates and asymptotics of the Green function, gave the Martin representation of harmonic functions, and resolved into sharp estimates of the heat kernel.
The development was initiated by Bogdan \cite{MR1438304} and Song and Wu \cite{MR1719233} with proofs of (BHP) for the fractional Laplacian.
Then Jakubowski \cite{MR1991120} gave sharp estimates of the Green function. Bogdan et al. \cite{MR2365478} gave  the boundary limits of ratios of harmonic functions and
Bogdan et al.  \cite{MR2722789} gave sharp estimates of the Dirichlet heat kernel.
Related works on the Dirichlet problem in cones are
given
by DeBlassie \cite{MR1062058}, Kulczycki \cite{ MR1750907}, Kulczycki and Burdzy \cite{MR1959790}, M{\'e}ndez-Hern{\'a}ndez \cite{MR1936081}, Bogdan and Jakubowski \cite{MR2182071}, Michalik \cite{MR2213639}, Kulczycki and Siudeja \cite{MR2231884}, and Bogdan and Grzywny \cite{MR2602155}.
For smooth domains we refer to the pioneering works by Kulczycki \cite{MR1490808}, Song and Chen \cite{MR1654824} and Kim et al. \cite{MR2994122, MR3232031}.
Historically, the results for cones often preceded and informed generalizations to Lipschitz and arbitrary open sets.
We expect similar generalizations for the asymptotics of  heat kernels. 
The present paper only resolves the asymptotics of the heat kernel of the fractional Laplacian in the  Lipschitz cone at the vertex, so there is much more work left to do.

The paper is organized as follows. In Section \ref{sec:prel} we give basic notation and facts.
In Section \ref{sec:main} we present our main results, which complement Theorem~\ref{thm:eYl} and Proposition~\ref{prop:qs}.
Most of the proofs are given in Section~\ref{sec:proofs}. In Section~\ref{sec:cauchy} we discuss in detail the Cauchy process on the positive half-line and we give a spectral decomposition of its Yaglom limit.
\section{Preliminaries}\label{sec:prel}

As defined in
the introduction, $X=\{X_t,\  t>0\}$ is the isotropic $\alpha$-stable L\'evy process
on the Euclidean space $\Rd$.
The process is determined by the jump measure with the density function
\begin{equation}
  \label{eq:lm}
  \nu(y)=\frac{2^{\alpha}\Gamma((d+\alpha)/2)}{\pi^{d/2}|\Gamma(-\alpha/2)|}
|y|^{-d-\alpha}\,, \quad y\in \Rd\,,
\end{equation}
where $0<\alpha<2$, $d=1,2,\ldots$. The coefficient in \eqref{eq:lm}
is
chosen so that
\begin{equation}
  \label{eq:trf}
  \int_\Rd \left[1-\cos(\xi\cdot y)\right]\nu(y)\rd y=|\xi|^\alpha\,,\quad
  \xi\in \Rd\,,
\end{equation}
for convenience.
Here $\xi\cdot y$ is the Euclidean scalar product and $|\xi|$ is the Euclidean norm.
We always assume in this paper that the considered sets, measures and
functions are %Borelian.
all Borel. 
The process $X$ is Markovian
with the following time-homogeneous transition probability
$$
P_t(x,A)=\int_A p_{t}(x-y)\rd y\,,\quad t>0\,,\;x\in \Rd\,,\; A\subset \Rd,
$$
and $p_t$ 	is the smooth real-valued function on $\Rd$ with the Fourier transform:
\begin{equation}
  \label{eq:dpt}
  \int_ \Rd p_t(x)\e^{\im x\cdot\xi}\,\rd x=\e^{-t|\xi|^\alpha}\,,\quad \xi\in
   \Rd\,.
\end{equation}
In particular, if $\alpha=1$, then $X$ is a Cauchy process, and we have
\begin{equation*}
p_t(x)=
\Gamma((d+1)/2)\pi^{-(d+1)/2}
\frac{t}{\big(|x|^2+t^2)^{(d+1)/2}}\,,
\end{equation*}
see \cite{MR2569321,MR0290095}.
For every $\alpha\in (0,2)$, the infinitesimal generator of $X$ is the fractional Laplacian,
\begin{equation}\label{eq:fL}
   \Delta^{\alpha/2}\varphi(x) =
\lim_{\varepsilon \downarrow 0}\int\limits_{|y|>\varepsilon}
   \left[\varphi(x+y)-\varphi(x)\right]\nu(y)
\rd y\,,
\quad
x\in \Rd\, ,
   \end{equation}
defined at least on smooth compactly supported functions $\phi\in C^\infty_c(\Rd)$,
cf. \cite{MR1671973, MR2569321, MR0481057, MR2013738, MR1873235, MR1739520, MR1336382, 2015arXiv150707356K}.
The following scaling property is a consequence of (\ref{eq:dpt}),
\begin{equation}
  \label{eq:sca}
  p_t(x)=t^{-d/\alpha}p_1(t^{-1/ \alpha}x)\,,\quad x\in \Rd\,,\;t>0\,.
\end{equation}
Furthermore,
\begin{equation}
  \label{eq:setd}
c^{-1} \left(\frac{t}{|x|^{d+\alpha}} \land
        t^{-d/\alpha}\right) \le p_t(x) \le c
        \left(\frac{t}{|x|^{d+\alpha}} \land t^{-d/\alpha}\right)\,,
\quad x\in  \Rd\,,\;t>0\,,
\end{equation}
see \cite{MR0119247, 2013arXiv13050976B, MR2013738} for the explicit constant $c$.
Below we will use the notation $f\approx g$ when functions $f$, $g\ge 0$ are {\it comparable} i.e. their ratio is bounded between two positive constants (uniformly on the whole domain of the functions).
In particular we can rewrite (\ref{eq:setd}) as follows:
\begin{equation}
  \label{eq:ppfe}
p_t(x)\approx t^{-d/\alpha}\wedge \frac{t}{|x|^{d+\alpha}}
\,, \quad x\in \Rd\,,\; t>0\,.
\end{equation}
We will also write $\lim f(x)/g(x)=1$ as $f(x)\sim g(x)$.

As stated, $\Gamma$ denotes a generalized  Lipschitz cone in $\R^d$ with vertex $0$, that is,
an open Lipschitz set $\Gamma \subset \R^d$ such that $0
\in \partial \Gamma$, and if $y \in \Gamma$ and $r
> 0$ then $r y \in \Gamma$.
Recall that an open set $D \subset \R^d$ is called Lipschitz
if there exist  $R > 0$ and
$\Lambda > 0$ such that for every $Q \in \partial D$, there exist a
Lipschitz function $\phi_Q$: $\R^{d - 1} \to \R$ with Lipschitz
constant not greater than $\Lambda$ and an orthonormal coordinate
system $\CS_Q$ such that if $y = (y_1, \ldots, y_{d - 1}, y_d)$ in
$\CS_Q$ coordinates, then
\[ D \cap B(Q, R) = \{y:\ y_d > \phi_Q(y_1, \ldots, y_{d - 1}) \} \cap
B(Q, R)\,. \]
We note that the trivial cones $\Gamma=\Rd$ and $\Gamma=\emptyset$ are excluded from our considerations because we require $0\in \partial \Gamma$, and the Lipschitz condition excludes, e.g., $\Rd\setminus\{0\}$. In particular for $d=1$, $\Gamma$ is necessarily a half-line. We note that the cone $\Gamma$ is
characterized by its intersection with the unit sphere $\S_{d - 1}
=\{x\in \R^d:|x|=1\}$.

The first exit time
from $\Gamma$, 
as defined in (\ref{deftau}), yields the heat kernel $p^\Gamma_t(x,y)$ of the cone,
\begin{equation}\label{Dirichlet}
p^\Gamma_t(x,y):=p_t(x,y)-\E_x[\tau_\Gamma<t;\,
p_{t-\tau_\Gamma}(X_{\tau_\Gamma},y)],\quad x\,,\; y\in \Rd,\,\; t>0\,,
\end{equation}
where $\E_x[\tau_\Gamma<t;\,
p_{t-\tau_\Gamma}(X_{\tau_\Gamma},y)]=\int\limits_{\{\tau_\Gamma<t\}}
p_{t-\tau_\Gamma}(X_{\tau_\Gamma},y)\rd \P_x$.
For bounded or nonnegative functions $f$ we have
$$P_t^\Gamma f(x):=\E_x[f(X^\Gamma_t)]=\E_x[\tau_\Gamma<t;\, f(X_t)]=\int_\Rd f(y)p^\Gamma_t(x,y)\rd y\,.$$
We also note that
\begin{equation}\label{eq:**}
0\leq p^\Gamma_t(x,y)=p_t^\Gamma(y,x)\leq p_t(y-x)\,,
\end{equation}
and $p^\Gamma$ satisfies the Chapman-Kolmogorov equations:
\begin{equation}\label{eq:CKpG}
\int p_s^\Gamma(x,y)p_t^\Gamma(y,z)\rd y=p_{t+s}^\Gamma(x,z),\quad
s\,,\; t>0\,,\quad x\,,\;z\in \Rd\,,
\end{equation}
see, e.g., \cite{MR1671973,Bogdan2016, MR2677618}. Since $\Gamma$ is Lipschitz, by the exterior cone condition and Blumenthal 0-1 law, $\mathbb{P}_x(\tau_\Gamma)=0$ if $x\in \Gamma^c$, in particular $p_t^\Gamma(x,y)=0$ whenever $x$ or $y$ are outside of $\Gamma$.
We note that
\begin{equation}\label{survival0}
\P_x(\tau_\Gamma>t)=\int_\Gamma p^\Gamma_t(x,y)\rd y\,.\end{equation}
From (\ref{eq:sca}), the following scaling property follows:
\begin{equation}
  \label{eq:scD}
  p^\Gamma_t(x,y)=t^{-d/\alpha}
  p_1^{t^{-1/\alpha}\Gamma}(t^{-1/\alpha}x,t^{-1/\alpha}y)\,,\quad
  x\,,\; y\in \Rd\,,\; t>0\,.
\end{equation}
By (\ref{survival0}), a similar scaling holds for the survival probability:
$$\P_x(\tau_\Gamma>t)=\P_{t^{-1/\alpha}x}(\tau_{t^{-1/\alpha}\Gamma}>1).$$
We define the Green function of $\Gamma$:
$$G_\Gamma(x,y)=\int_0^\infty p_t^\Gamma(x,y)\rd t,\quad x\,,\; y\in \Rd\,,$$
and the Green operator:
$$G_\Gamma f(x)=\int_{\Gamma}G_\Gamma(x,y)f(y)\;\rd y\,,$$
for integrable or nonnegative functions $f$.
Clearly, by (\ref{eq:**}),
\begin{equation}\label{eq:supG}
P_t^\Gamma \left( G_\Gamma(\cdot,y)\right) (x)\le
G_\Gamma(x,y)\,,\quad t>0\,, \; x\,,\; y\in \Rd\,.
\end{equation}
We should note that $G_\Gamma$ is always locally integrable because $\Gamma\neq \mathbb{R}$ and $\Gamma\neq\mathbb{R}\setminus\{0\}$, cf. \cite{MR2256481}.

For $r\in (0,\infty)$ we let $B_r = \{ |x| < r\}$ and we define the truncated cone
$$\Gamma_r = \Gamma \cap B_r\,.$$
By the strong Markov property, for $t>0$ and $x$, $y \in {\Rd}$,
\begin{equation}
  p^{\Gamma}_t(x,y) = p^{\Gamma_{r}}_t(x, y) + \E_x \left[
    \tau_{\Gamma_{r}} < t;\,
    p^{\Gamma}_{t - \tau_{\Gamma_{r}}} (X_{\tau_{\Gamma_{r}}}, y) \right]\,.
\end{equation}
Integrating the above identity against $\rd t$ (see \cite[eq. (15)]{MR2365478}), we obtain
\begin{equation}
  \label{e:GGamma2}
  G_{\Gamma}(x, y) = G_{\Gamma_{r}}(x, y) + \E_x \left[
    G_{\Gamma}(X_{\tau_{\Gamma_{r}}}, y) \right]\,.
\end{equation}
In particular, $x\to G_\Gamma(x,y)$
is regular $\alpha$-harmonic on $\Gamma_r$ if $|y| > r$.
We recall that a function $u: \R^d\to \R$ is called regular harmonic
with respect to $X$ on an open set $\Gamma\subset \R^d$ if
\[ u(x) = \E_x \left[ \tau_\Gamma < \infty;\, u(X_{\tau_\Gamma}) \right]\,, \quad x
 \in \Gamma\,, \]
 where we assume that the expectation is finite.

Two more facts are crucial in our development.
First,
if $I\subset (0,\infty)$
and $A\subset \Gamma$, $B\subset (\overline{\Gamma})^c$, then
\begin{align}\label{eq:IW}
\P_x[\tau_\Gamma\in I,\; X_{\tau_\Gamma-}\in A,\; X_{\tau_\Gamma}\in B]=
\int_I \int_{B-y} \int_A p_u^
\Gamma(x,\rd y)\nu(z)\rd z \rd u\,.
\end{align}
This identity is called the Ikeda-Watanabe formula \cite{MR0142153}, and gives the joint distribution of $(\tau_\Gamma,X_{\tau_\Gamma-},X_{\tau_\Gamma})$,
see also \cite[Lemma 1]{MR2345912}, \cite{1411.7952}, \cite[Appendix A]{MR2357678}, \cite[VII.68]{MR745449}
or \cite[Theorem 2.4]{MR2255353}.

Second, we use the following boundary Harnack principle (BHP) from Bogdan \cite{MR1438304}:
There is a constant $C = C(\Gamma, \alpha) > 0$ such that
  if $r>0$ and functions
  $u, v\ge 0$ are regular harmonic in $\Gamma_{2 r}$ with respect to $X$  and vanish on
  $\Gamma^c \cap B_{2 r}$, then
  \begin{equation}
    \label{e:BHP}\tag{BHP}
    u(x)v(y)\leq C u(y)v(x)\,, \quad x\,,\; y \in
    \Gamma_r\,.
  \end{equation}
  Generalizations of \eqref{e:BHP} can be found in \cite{MR1719233, MR2365478} for the fractional Laplacian
and in \cite{MR3271268, MR2994122, MR3232031} for more general jump Markov processes.
By an oscillation-reduction argument, \eqref{e:BHP} implies that $\lim_{\Gamma\ni x\to 0} u(x)/v(x)$ exists \cite{MR2075671, MR2365478}.
Without essential loss of generality, in what follows we assume that 
$$\id := (0, \ldots, 0, 1) \in
\Gamma\,.$$
The above 
directly  implies the existence of
\begin{equation}
  \label{e:martini}
  M(y) = \lim_{\Gamma \ni x, |x|\to \infty} \frac{G_{\Gamma}(x,y)}{G_{\Gamma}(x, \id)}\,, \quad y\in \R^d\,.
\end{equation}
$M$ is called the Martin kernel with the pole at infinity for $\Gamma$. It is the unique
nonnegative function on $\R^d$
that is
regular harmonic
with respect to $X$ on every $\Gamma_r$
and such that $M = 0$
on $\Gamma^c$ and $M(\id) = 1$ \cite[Theorem 3.2]{MR2075671}.
The function is locally bounded on $\R^d$ and
homogeneous of degree $\beta = \beta(\Gamma, \alpha)$, that is,
 \begin{equation}\label{eq:hoM}
M(x) = | x |^{\beta} M(x / |x|)\,, \quad x \neq 0\,.
\end{equation}
Furthermore,
$0  < \beta < \alpha$.
The exponent $\beta$ is decreasing in $\Gamma$ and it delicately
depends on the geometry of $\Gamma$. When $\Gamma$ is a right-circular
cone, a rather explicit estimate
for $M$ is available \cite[Theorem~3.13]{MR2213639}, expressed in terms of $\beta$.  More information on $\beta$ for narrow right-circular cones is given in \cite{MR3339224}.
As we shall see below, using \eqref{e:BHP} and $M$ we can capture the boundary asymptotics of harmonic functions and some Green potentials. 

Following \cite{MR2075671} we consider the Kelvin transformation  $K$ of $M$:
 \begin{equation}
   \label{e:kelvin}
   K(x) = |x|^{\alpha - d} M(x/|x|^2) = |x|^{\alpha - d - \beta}
   K(x/|x|)\,, \quad x\neq 0\,,
 \end{equation}
see also Bogdan and {\.Z}ak \cite{MR2256481} for a more general discussion. The function $K$
is called the Martin kernel at $0$ for $\Gamma$. Clearly, $K(\id) = 1$ and $K = 0$ on $\Gamma^c$.
By \cite[Theorem 3.4]{MR2075671},
 \[ K(x) = \E_x \left[\tau_B < \infty;\, K(X_{\tau_B}) \right]\,, \quad x
 \in \R^d\,, \]
for every open set $B \subset \Gamma$ with $\dist(0, B) > 0$.
In particular $K$ is $\alpha$-harmonic in $\Gamma$ and satisfies
\begin{equation}
  \label{e:martin}
  K(y) = \lim_{\Gamma \ni x \to 0} \frac{G_{\Gamma}(x,y)}{G_{\Gamma}(x, \id)}\,, \quad y\in \R^d\setminus\{0\}\,.
\end{equation}

\section{Full picture}\label{sec:main}
Theorem~\ref{thm:eYl} and Proposition~\ref{prop:qs} are manifestations of phenomena
which we present 
later in this section.

By \eqref{e:BHP}, a finite positive limit
\begin{equation}\label{eq:C0}
C_0 = \lim_{\Gamma \ni x \to 0} G_{\Gamma}(x, \id) / M(x)
\end{equation}
exists.
We denote
\begin{equation}\label{defkappa}\kappa_{\Gamma}(z) = \int_{\Gamma^c} \nu(z - y) \rd y\,,
\end{equation}
where $\nu$ is a jump measure defined in (\ref{eq:lm}),
and we define
  \begin{equation}\label{Cjeden} C_1 = C_0 \int_{\Gamma} \int_{\Gamma}
      K(y) p^{\Gamma}_1(y,z)
  \kappa_{\Gamma}(z) \rd z \rd y\,.  \end{equation}

\begin{thm}\label{T:survival}
We have $0<C_1<\infty$ and
  \begin{equation}
    \label{e:survival1}
  \lim_{\Gamma \ni x \to 0} \frac{\P_x (\tau_{\Gamma} > 1)}{M(x)} =
  C_1\,.
  \end{equation}
  \end{thm}
The proofs of Theorem~\ref{T:survival} and the other results of this section are mostly deferred to Section~\ref{sec:proofs}, to allow for a streamlined presentation.

Theorem~\ref{T:survival}, the scaling property of $X$  and the $\beta$-homogeneity of $M$, that is
(\ref{eq:hoM}) and (\ref{eq:sca}), yield the following result, which refines Lemma 4.2 of Ba{\~n}uelos and Bogdan \cite{MR2075671}.
\begin{cor}\label{assurv}
Uniformly as $\Gamma \ni t^{-1/\alpha} x \to 0$ we have
  \begin{equation}
    \label{e:survivalt}
    \P_x(\tau_{\Gamma} > t) \sim C_1 M(x) t^{- \beta /\alpha}\,.
  \end{equation}
\end{cor}
Another consequence of Theorem~\ref{T:survival} is the following theorem, which is the main result of the paper.
\begin{thm}\label{mainresult}
The following limit exits,
  \begin{equation}
    \label{e:blimit}
    n_t(y) = \lim_{\Gamma \ni x \to 0} \frac{p^{\Gamma}_t(x,
      y)}{\P_x(\tau_{\Gamma} > 1)}\,, \quad (t, y) \in (0, \infty) \times \Gamma\,.
  \end{equation}
It is a finite strictly positive jointly continuous function of $t$ and $y$, and we have
\begin{equation}\label{QSdensity}
\mu(A)=\int_A n_1(z)\;\rd z\,,\quad A\subset \Gamma\,.
\end{equation}
Furthermore, for $0<s,t<\infty$, $y \in \Gamma$,
  \begin{align}
    \label{e:scaling}
    n_t(y) &= t^{-(d + \beta)/\alpha}
    n_1(t^{-1/\alpha}y)\,,\\
 \label{e:mubound}
  n_1(y) &\approx \frac{\P_y(\tau_{\Gamma} > 1)}{(1 + |y|)^{d +
      \alpha}}\,, \\
    \label{e:ck}
   n_{t+s}(y)  &= \int_{\Gamma}n_t(z)p^{\Gamma}_s(z, y) \rd z\,.
  \end{align}
\end{thm}
In view of \eqref{e:ck}, $n_t(y)\rd y$ defines an entrance law of
excursions {surviving at least time $t$}
from $0$ into $\Gamma$, cf. Rivero and Haas \cite{MR2971725}, Blumenthal \cite[page 104]{MR1138461} and
Ba{\~n}uelos et al. \cite{MR1042064}.

We note that Yano \cite{MR2603019} studies excursions of symmetric L\'evy processes into $\R\setminus\{0\}$, a situation not discusses in this paper. We however note that in our situation 
$$\int_\Gamma n_t(x)\rd x=t^{-\beta/\alpha},\ \quad t>0\ ,$$ which nicely corresponds with \cite[Example~1.1]{MR2603019}, because for $\Gamma=\R\setminus\{0\}$ and $\alpha\in (1,2)$, $\beta=\alpha-1$, see \cite{MR2075671}. 
\begin{example}
If $d=1$, $\Gamma$ is the half-line $(0,\infty)$, 
then $M(x)=x^{\alpha/2}$ for $x>0$ \cite[Example~3.2]{MR2075671}.
By \cite[Theorem~2]{MR2722789},
$$
\P_x(\tau_\Gamma>1)\approx x^{\alpha/2}\wedge 1 , \quad x>0.
$$ 
By \eqref{e:mubound},
\begin{equation*} 
n_1(x)\approx x^{\alpha/2} \wedge x^{-d-\alpha}, \quad x>0.
\end{equation*}
Therefore by \eqref{e:scaling},
\begin{equation}\label{eq:onthl}
n_t(x)\approx (x^{\alpha/2}t^{-1-1/\alpha})\wedge (x^{-1-\alpha}t^{1/2}), \quad t,x>0.
\end{equation}
The first expression gives the minimum if $x^\alpha\leq t$ (small space), and the second  -- if $x^\alpha>t$ (short time).
Estimates of $n_t(x)$ for half-spaces in $\Rd$ may be obtained in a similar way. 
 \end{example}

Some of the other objects we study can also be expressed in terms of
$n$. 
Namely we have
    \begin{equation}
      \label{e:KC0C1}
      \begin{aligned}
      K(y) = & \lim_{\Gamma \ni x \to 0}
      \frac{G_{\Gamma}(x,y)}{\P_x(\tau_{\Gamma}>1)}
      \frac{\P_x(\tau_{\Gamma}>1)}{M(x)}
      \frac{M(x)}{G_{\Gamma}(x,\id)} \\
      = & \frac{C_1}{C_0} \lim_{\Gamma
        \ni x \to 0} \int_0^{\infty}
      \frac{p^{\Gamma}_t(x,y)}{\P_x(\tau_{\Gamma}>1)} \rd t \\
      = & \frac{C_1}{C_0} \int_0^{\infty} n_t(y) \rd t\,.
    \end{aligned}
  \end{equation}
Therefore,
    \begin{equation}
      \label{e:C1C0}
      \frac{C_0}{C_1} = \int_0^{\infty} n_t(\id) \rd t
    \end{equation}
may be interpreted as the expected amount of time spent at ${\bf 1}$ by the excursion from the vertex into $\Gamma$.

We %like to
note in passing that the spatial asymptotics of the heat kernel at infinity
were given in the works of Blumenthal and Getoor \cite{MR0119247, MR0264757}
(see also \cite{2015arXiv150408358C}), who showed that
$p_t(x)\sim
t\nu(x)$ as
$t|x|^{-\alpha}\to 0$.
More results of this type for unimodal L\'evy processes can be found in recent works of Tomasz Grzywny et al., including \cite{2015arXiv150408358C}, however, the above papers only concern $\Gamma=\Rd$.

Our approach to Theorem~\ref{mainresult} depends on three properties. First, the scaling
(\ref{eq:sca}) yields
\begin{equation}\label{repr1}
\P_x \left(  \tau_{\Gamma} > t,\frac{X_t}{t^{1/\alpha}} \in A \right)= \P_{t^{-1/\alpha} x} \left(\tau_{\Gamma} > 1, X_1 \in A\right)
\end{equation}
and
\begin{equation}\label{repr2}
\P_x \left(  \tau_{\Gamma} > t\right)= \P_{t^{-1/\alpha} x} \left(\tau_{\Gamma} > 1\right)\;.
\end{equation}
Then the Ikeda-Watanabe formula (\ref{eq:IW}) gives the representation
\begin{equation}\label{lem:rsp}
\P_x (\tau_{\Gamma} > 1)=G_\Gamma P^\Gamma_1 \kappa_\Gamma(x)\,.
\end{equation}
Recall that $\kappa_\Gamma(x)$ may be considered as the killing intensity because it is the intensity of
jumps of $X$ from $x$ to $\Gamma^c$. Similarly, $P^\Gamma_1
\kappa_\Gamma(x)$ may be interpreted as the intensity of killing
precisely
one unit of time from now.

To actually prove the existence of $n_t$ in (\ref{e:blimit}),
we use the asymptotics of Green potentials at the vertex $0$.
\begin{lemm}
  \label{T:Greenpotential}
If $f$ is a measurable function on $\Gamma$, bounded on $\Gamma_1$ and $G_{\Gamma} |f| (\id) < \infty$, then
\begin{equation}
  \label{e:Greenpotential}
    \lim_{\Gamma \ni x \to 0} \frac{G_{\Gamma} f(x)}{M(x)} = C_0
    \int_{\Gamma} K(y) f(y) \rd y<\infty\,.
  \end{equation}
\end{lemm}

\section{Proofs}\label{sec:proofs}
\subsection{Proof of Lemma~\ref{T:Greenpotential}}\label{s:greenpotential}
If $G_{\Gamma} |f| (x) < \infty$ for some $x \in \Gamma$,
then by \cite[Lemma 5.1]{MR1825645}, $G_{\Gamma} |f|(x) <
\infty$ for almost all $x \in \R^d$.
Let $0<\delta<1$.
Choose $x_1 \in \Gamma_{
{\delta/2}}$ so that $G_{\Gamma} |f|
(x_1) < \infty$.
  By \eqref{e:BHP},
  \begin{equation}\label{eq:BHPG}
    \frac{G_{\Gamma}(x,y)}{G_{\Gamma}(x,\id)} \leq c_1
    \frac{G_{\Gamma}(x_1,y)}{G_{\Gamma}(x_1, \id)}\,,
\quad
x\,,\; y \in \Gamma\,,\quad |x| < \delta/2\,,\quad  |y| \geq \delta
  \end{equation}
for some constant $c_1$.
By \eqref{e:martin}, \eqref{eq:BHPG}, and Fatou's lemma we see that the right-hand side of (\ref{eq:BHPG}) is finite.
Observe that we also have
  \begin{equation}\label{eq:fx1}
    \int_{\Gamma \setminus \Gamma_{\delta}} G_{\Gamma}(x_1, y) |f(y)| \rd
    y \leq \int_{\Gamma} G_{\Gamma}(x_1, y) |f(y)| \rd y < \infty\,.
  \end{equation}
  By \eqref{e:martin}, {\eqref{eq:BHPG}, \eqref{eq:fx1}} and the dominated convergence theorem,
  \begin{equation}
    \label{e:Gammalarge}
    \lim_{\Gamma \ni x \to 0} \int_{\Gamma \setminus \Gamma_{\delta}}
    \frac{G_{\Gamma}(x,y)}{G_{\Gamma}(x,\id)} f(y) \rd y =
    \int_{\Gamma \setminus \Gamma_{\delta}} K(y) f(y) \rd y\,.
  \end{equation}
  % Now
  Next we consider the %integration
  integral over $\Gamma_{\delta}$. By our assumptions on $f$, a change of variables, and the scaling
property $G_{\Gamma}(\delta x, \delta y) = \delta^{-d + \alpha}
G_{\Gamma}(x, y)$ we can conclude that for some constant $c_2$,
\begin{align}
\int_{\Gamma_{\delta}} G_{\Gamma}(x, y) |f(y)| \rd y
  & \leq  c_2 \delta^d \int_{\Gamma_1} G_{\Gamma}(x, \delta z)
      \rd z
= c_2 \delta^{\alpha} \int_{\Gamma_1} G_{\Gamma}(\delta^{-1} x,
  z) \rd z\,, \quad x\in \Gamma_\delta\,.  \label{e:Gammadelta}
\end{align}
Now, by \eqref{e:BHP}, for some contant $c_3$,
\begin{equation}
  \label{e:GGammavy}
  \frac{G_{\Gamma}(v, y)}{M(y)} \leq c_3 \frac{G_{\Gamma}(v,
    \id)}{M(\id)} = c_3 G_{\Gamma}(v, \id), \quad v \in \Gamma
  \setminus \Gamma_2\,, \; y \in \Gamma_1\,.
\end{equation}
Indeed, by the symmetry of $G_\Gamma$, the regular harmonicity of $y\mapsto G_{\Gamma}(v, y)$ {on $\Gamma_{1}$} follows from
\eqref{e:GGamma2}. Furthermore, the continuity of $\alpha$-harmonic
functions allows us to use $\id$ in \eqref{e:GGammavy}.
By \eqref{e:GGamma2} and \eqref{e:GGammavy} we have
\begin{align}
  G_{\Gamma}(x, y)
  = & G_{\Gamma_2}(x, y) + \E_x \left[ G_{\Gamma}(X_{\tau_{\Gamma_2}},
      y) \right]
  \leq  G_{\Gamma_2}(x, y) + c_3 \E_x \left[
         G_{\Gamma}(X_{\tau_{\Gamma_2}}, \id) \right] M(y) \nonumber\\
  \leq & G_{\Gamma_2}(x, y) + c_3 G_{\Gamma}(x, \id) M(y)\,, \quad x\,,\; y \in \Gamma_1\,. \label{eq:ubf}
\end{align}
By identities \eqref{e:Gammadelta}, \eqref{eq:ubf} and the local boundedness of
$M$, we have
\begin{equation}
  \label{e:Gammadelta2}
  \int_{\Gamma_{\delta}} G_{\Gamma}(x, y) |f(y)| \rd y \leq c_2
  \delta^{\alpha} \int_{\Gamma_1} G_{\Gamma_2}(\delta^{-1}x, z) \rd z
  + c_3 \delta^{\alpha} G_{\Gamma}(\delta^{-1}x, \id)
\end{equation}
for every $x\in \Gamma_\delta$.
Let
\[ c_4 = \inf_{z \in \Gamma_1} \int_{\Gamma \setminus \Gamma_2}\nu(z-w)
\rd w\,. \]
Clearly, $c_4 > 0$. By the Ikeda--Watanabe formula \eqref{eq:IW} and \eqref{e:BHP}, {for $x\in \Gamma_\delta$,}
\begin{align*}
  \int_{\Gamma_1} G_{\Gamma_2}(\delta^{-1} x, z) \rd z
  \leq & c_4^{-1}
\int_{\Gamma \setminus \Gamma_2} \int_{\Gamma_1}
G_{\Gamma_2}(\delta^{-1}x, z) \frac{\sA_{d,\alpha}}{|w-z|^{d +
         \alpha}} \rd z \rd w \\
  \leq & c_4^{-1} \P_{\delta^{-1} x} \left(
         X_{\tau_{\Gamma_2}} \in \Gamma \right) \\
  \leq & c_5 G_{\Gamma}(\delta^{-1}x, \id)\,.
\end{align*}
Again by \eqref{e:BHP} we have
\[ G_{\Gamma}(\delta^{-1}x, \id) \approx M(\delta^{-1} x) =
\delta^{-\beta} M(x) \approx \delta^{-\beta} G_{\Gamma}(x, \id) \]
for $x \in \Gamma_{\delta/2}$.
In view of \eqref{e:Gammadelta2},
\begin{equation}
  \label{e:Gammasmall}
  \int_{\Gamma_{\delta}} \frac{G_{\Gamma}(x, y)}{G_{\Gamma}(x, \id)}
  |f(y)| \rd y \leq c_6 \delta^{\alpha - \beta}\,, \quad x \in \Gamma_{\delta/2}\,.
\end{equation}
From \eqref{e:Gammalarge}, {Fatou's lemma} and \eqref{e:Gammasmall} it follows that
  %\begin{equation}
  \[  \int_{\Gamma \setminus \Gamma_{\delta}} K(y) f(y) \rd y \leq
    \liminf_{\Gamma \ni x \to 0} \frac{G_{\Gamma}
      f(x)}{G_{\Gamma}(x,\id)} \leq \limsup_{\Gamma \ni x \to 0}
    \frac{G_{\Gamma} f(x)}{G_{\Gamma}(x,\id)} \leq \int_{\Gamma
      \setminus \Gamma_{\delta}} K(y) f(y) \rd y + c_6 \delta^{\alpha
      - \beta}\,. \]
  %\end{equation}
  Taking the limit in the above identity as $\delta \to 0$ and using
  the fact that $\alpha > \beta$, we establish that
  \begin{equation}
    \lim_{\Gamma \ni x \to 0} \frac{G_{\Gamma}f(x)}{G_{\Gamma}(x,\id)}
    = \int_{\Gamma} K(y) f(y) \rd y\,.
  \end{equation}
  We then apply \eqref{eq:C0}, %and complete
  which completes the proof.
\halmos

\subsection{Proof of (\ref{lem:rsp})}
We observe that  by the Lipschitz condition and Sztonyk \cite[Theorem~1]{MR1825650} we have
$\P_x \left( X_{\tau_{\Gamma}} \in \partial \Gamma \right)=0$ for every $x \in \Gamma$. Thus,
\begin{equation}
\label{e:jumpexit}
\P_x \left( X_{\tau_{\Gamma} -} = X_{\tau_{\Gamma}} \right) = 0\,.
\end{equation}
Now for $x\in \Gamma^c$ we have $\tau_\Gamma=0$ $\P^x$-a.s. and (\ref{lem:rsp}) holds true.
To prove (\ref{lem:rsp}) for $x\in \Gamma$, observe that by \eqref{eq:IW} and \eqref{e:jumpexit},
\begin{align*}
  \label{e:survival}
  \P_x (\tau_{\Gamma} > 1)= & \int_1^{\infty} \int_{\Gamma} p^{\Gamma}_s(x,z)
  \kappa_{\Gamma}(z) \rd z \rd s \\
  = & \int_0^{\infty} \int_{\Gamma} \int_{\Gamma}
    p^{\Gamma}_s(x, w) p^{\Gamma}_1(w, z) \rd w \,
  \kappa_{\Gamma}(z) \rd z \rd s\\
 = &G_\Gamma P^\Gamma_1 \kappa_\Gamma(x)\,.
\end{align*}
\halmos

\subsection{Proof of Theorem \ref{T:survival}}
\label{s:exit}

By \cite{MR1490808} we have $G_\Gamma(x,w)>0$ for all $x,w\in \Gamma$.
Thus $P^\Gamma_1 \kappa_\Gamma$ is finite almost everywhere. Furthermore,
the following estimate holds:
\begin{equation}
P^\Gamma_1 \kappa_\Gamma(x) \approx \P_x
  (\tau_{\Gamma} > 1)\quad \text{for }x \in \Gamma_1\,.  \label{L:fGamma}
\end{equation}
Indeed, if $|x| \leq 1$, then by \eqref{eq:ppfe},
\begin{equation}
  \label{e:p1xy}
p_1(x, y) \approx 1 \wedge |x-y|^{-d - \alpha} \approx ( 1 +
  |y|)^{-d - \alpha}\,.
\end{equation}
Furthermore, by \cite{MR2602155} or
\cite[Theorem 2]{MR2722789}, the following approximate factorization holds
  \begin{equation}
    \label{e:p1Gamma}
p^{\Gamma}_1(x,y) \approx \P_x (\tau_{\Gamma} > 1)
  \P_y(\tau_{\Gamma} > 1) p_1(x,y)\,, \quad x\,, \;y \in \Gamma\,.
\end{equation}
Thus,
  \begin{equation}\label{integralfinite}
  P^\Gamma_1 \kappa_\Gamma(x) \approx \P_x(\tau_{\Gamma} > 1) \int_{\Gamma}
  \P_y(\tau_{\Gamma} > 1) \left( 1 + |y| \right)^{-d - \alpha}
  \kappa_{\Gamma}(y) \rd y, \quad {x\in \Gamma_1}\,, \end{equation}
which finishes the proof of (\ref{L:fGamma}).
Note also that the integral on the right-hand side  of (\ref{integralfinite}) is finite.
This allows us to establish the
asymptotic behavior of $\P_x(\tau_{\Gamma} > 1)$.
The estimate \eqref{L:fGamma} ensures that $f(x)=P^\Gamma_1 \kappa_\Gamma(x)$ satisfies the
  assumptions of Lemma~\ref{T:Greenpotential}. Thus
  \eqref{e:survival1} is a direct consequence of
the representation \eqref{lem:rsp} and the identity \eqref{e:Greenpotential}.
This completes the proof.
\halmos

\subsection{Proof of Theorem \ref{mainresult}}
Consider the family
of measures $\{\mu_x:\, x \in \Gamma_1\}$ defined by
\begin{equation}
  \label{e:mux}
  \mu_x (A) = \frac{\int_A p^{\Gamma}_1(x, y) \rd
  y}{\P_x(\tau_{\Gamma} > 1)}\,, \quad x\in \Gamma\,,\;  A\subset \Rd\,.
\end{equation}
We start %from
by proving that the above family of measures
is tight.
Indeed, by \eqref{e:p1Gamma} and \eqref{e:p1xy} we can bound the
density of $\mu_x$ by the integrable function:
  \begin{equation}
    \label{e:pGamma1}
\frac{p^{\Gamma}_1(x, y)}{\P_x(\tau_{\Gamma} > 1)} \approx
\P_y(\tau_{\Gamma} > 1) (1 + |y|)^{-d - \alpha}\,,
\quad x \in \Gamma_1\,,\; y \in \R^d\,.
\end{equation}

We will prove now that the measures $\mu_x$ converge weakly to a probability measure $\mu$ on $\Gamma$
as $\Gamma\ni x\to 0$:
\begin{equation}\label{T:yaglom}
\mu_x\Rightarrow \mu\qquad \text{as $\Gamma\ni x\to 0$}\,.
\end{equation}
To prove \eqref{T:yaglom}, we consider an arbitrary sequence $\{x_n\}$ such that $\Gamma_1
\ni x_n \to 0$. By the tightness of the family of measures $\{\mu_x: x \in \Gamma_1\}$ there
exists a subsequence $\{x_{n_k}\}$ such that $\mu_{x_{n_k}}$ converge weakly to a probability measure $\mu$ as $k
\to \infty$.

Let $\phi \in C_c^{\infty}(\Gamma)$ and $u_{\phi} = -
\Delta^{\alpha/2} \phi$. The function $u_{\phi}$ is bounded and continuous and
$G_{\Gamma} u_{\phi} (x) = \phi(x)$ for $x \in
\R^d$, see \cite[Eq. (19)]{MR2892584} and \cite[Eq. (11)]{MR2365478} for
more details. By \eqref{eq:fL} we have $| u_{\phi}(x)| \leq c_1 p_1(x, 0)$. Then,
\begin{equation}\label{eq:p2}
P_1^{\Gamma} |u_{\phi}|(x)\le c_1 p_2(x)\,,
\end{equation}
 and for
every $x \in \Gamma$,
\[ G_{\Gamma} P^{\Gamma}_1 |u_{\phi}| (x) \leq \int_{\R^d} \int_{\R^d}
G_{\Gamma}(x, y) p_1(y, z) |u_{\phi}(z)| \rd z \rd y \leq c_1
\int_{\R^d} G_{\Gamma}(x, y) p_2(y, 0) \rd y < \infty\,, \]
see  \cite[Eq. (74)]{MR2365478}.
By Fubini's theorem,
\begin{equation}
  \label{e:GP1}
  G_{\Gamma} P_1^{\Gamma} u_{\phi} (x) = P_1^{\Gamma} G_{\Gamma}
  u_{\phi} (x) = P_1^{\Gamma} \phi(x)\,.
\end{equation}
It follows from Lemma~\ref{T:Greenpotential} and \eqref{eq:p2} that
\[ \lim_{\Gamma \ni x \to 0} \frac{P_1^{\Gamma} \phi(x)}{M(x)} =
\lim_{\Gamma \ni x \to 0} \frac{G_{\Gamma} P_1^{\Gamma}
  u_{\phi}(x)}{M(x)} = C_0 \int_{\Gamma} K(y) P_1^{\Gamma} u_{\phi}
(y) \rd y\,. \]
Denoting $\mu_x(\phi)=\int_\Gamma \phi(y)\;\mu_x(\rd y)$ and applying Theorem~\ref{T:survival}
we get a finite limit
\[ \lim_{\Gamma \ni x \to 0} \mu_x(\phi) = \lim_{\Gamma \ni x \to 0}
\frac{P_1^{\Gamma} \phi(x)}{\P_x(\tau_{\Gamma} > 1)} =
\frac{\int_{\Gamma} K(y) P_1^{\Gamma} u_{\phi} (y) \rd
  y}{\int_{\Gamma} K(y) P_1^{\Gamma} \kappa_{\Gamma}(y) \rd y}\,. \]
In particular, 
$\mu(\phi) = \lim_{k \to \infty} \mu_{x_{n_k}}(\phi)$
does not depend on the choice of the subsequence $\{x_{n_k}\}$.
Thus, $\mu_x$ weakly converges to this $\mu$ as $\Gamma \ni x \to 0$.

We are now in a position to prove that $n_1$ is the density of the Yaglom limit $\mu$ appearing in \eqref{T:yaglom} and that $n_t$ is well-defined.
By the Chapman-Kolmogorov equation applied to
$\phi_y(\cdot) = p^{\Gamma}_1(\cdot, y)\in C_0(\Rd)$,
  \begin{equation}
    \label{e:p2Gamma}
  p^{\Gamma}_2(x, y) = \int_{\Gamma} p^{\Gamma}_1(x, z)
p^{\Gamma}_1(z, y) \rd z = P^{\Gamma}_1 \phi_y (x)\,, \quad x\,,\; y \in
\Gamma\,.
\end{equation}
Thus, for all $y \in \Gamma$,
\begin{equation}
  \label{e:p2}
  \frac{p^{\Gamma}_2(x, y)}{\P_x(\tau_{\Gamma}>1)} = \frac{P^{\Gamma}_1\phi_y(x)}{\P_x(\tau_{\Gamma} > 1)} =
   \mu_x(\phi_y) \to
  \mu(\phi_y)<\infty\,,
\end{equation}
as $\Gamma\ni x\to 0$, see Theorem~\ref{T:survival} and \eqref{T:yaglom}.
This proves the existence of the limit $n_t$ defined in \eqref{e:blimit} for $t=2$.
Using the existence of this limit, the scaling property and Theorem~\ref{T:survival} we can conclude now that
for any $(t, y) \in (0, \infty) \times \Gamma$ the following holds true
\begin{equation}
  \label{e:nty}
  \begin{aligned}
  n_t(y) = & \lim_{\Gamma \ni x \to 0} \frac{p^{\Gamma}_t(x,
             y)}{\P_x(\tau_{\Gamma} > 1)} \\
  = &   (t/2)^{-d/\alpha}
  \lim_{\Gamma \ni x \to 0} \frac{ p^{\Gamma}_2
  \left( (t/2)^{-1/\alpha} x, (t/2)^{-1/\alpha}
      y \right)}{\P_{ \left( t / 2 \right)^{- 1/\alpha}
        x}(\tau_{\Gamma}> 1) } \lim_{\Gamma \ni x \to 0}
  \frac{\P_x(\tau_{\Gamma} > t/2)}{\P_x(\tau_{\Gamma} > 1)} \\
  = &   (t/2)^{-(d+\beta)/\alpha} n_2((t/2)^{-1/\alpha}y)\,.
\end{aligned}
\end{equation}
This proves the existence of the limit $n_t(y)$ for general $t>0$, and the equation \eqref{e:scaling}.
By \eqref{e:pGamma1} we get
\eqref{e:mubound}.
By the weak convergence \eqref{T:yaglom}, Theorem~\ref{T:survival},
and the dominated convergence theorem, we get that for
every bounded continuous function $\phi$ on $\Gamma$ we have
\[ \mu(\phi) = \lim_{\Gamma \ni x \to 0} \frac{P^{\Gamma}_1 \phi
    (x)}{\P_x \left( \tau_{\Gamma} > 1 \right)} = \lim_{\Gamma \ni x
    \to 0}
  \int_{\Gamma} \frac{p^{\Gamma}_1(x, y)}{\P_x \left( \tau_{\Gamma} > 1 \right)} \phi(y) \rd y =
  \int_{\Gamma} n_1(y) \phi(y) \rd y\,. \]
This completes the proof of the fact that the limit $n_1(y)$ from \eqref{mainresult} is well-defined and gives the density function of the quasi-stationary measure $\mu$.
Note that \eqref{e:ck} follows directly from
the Chapman-Kolmogorov equation and the dominated convergence theorem:
\[
n_{t+s}(y) = \lim_{\Gamma \ni x \to 0} \int_{\Gamma}
      \frac{p^{\Gamma}_t(x, z)}{\P_x(\tau_{\Gamma} > 1)}
      p^{\Gamma}_s(z, y) \rd z = P^{\Gamma}_s
    n_t(y)\,.
\]
To end the proof we show that $n_t(y)$ is jointly continuous 
on $(0,\infty)\times \Gamma$. Indeed,  $p_1^\Gamma(z,y)/p_1^\Gamma(z,y_1)\approx 1$ for $z\in \Gamma$,
if $y,y_1\in \Gamma$ are close to each other.
The continuity of $n_2(y)$ follows from the dominated convergence theorem and the continuity of $p^\Gamma_1$. The joint continuity of $n_t(y)$ follows from the scaling property.
\halmos

\subsection{Relatively uniform convergence}
\begin{lemm}\label{l:ru}
If 
$0\le c^{-1}f_n\le f_m\le c f_n$ for all $m,n$, and $f=\lim f_n$, then
$\lim \int f_n \rd\eta= \int f \rd\eta$.  
\end{lemm}
This is true because if the integral $ \int f \rd\eta$ is finite, then the dominated convergence theorem applies.

\subsection{Proof of Theorem \ref{thm:eYl}}
\label{s:yaglom}
By the scaling property of $X_t$ we have
\begin{align*}
  \P_x \left( \frac{X_t}{t^{1/\alpha}} \in A | \tau_{\Gamma} > t
  \right)
  = & \frac{\P_x \left(  \tau_{\Gamma} >
  t,\frac{X_t}{t^{1/\alpha}} \in A \right)}{\P_x(\tau_{\Gamma} > t)}
  =  \frac{\P_{t^{-1/\alpha} x} \left(\tau_{\Gamma} > 1, X_1 \in A
  \right)}{\P_{t^{-1/\alpha} x}(\tau_{\Gamma} > 1)} \\
  = & \frac{\int_A p^{\Gamma}_1(t^{-1/\alpha}x, y) \rd
      y}{\P_{t^{-1/\alpha} x} (\tau_{\Gamma} > 1)} %=
    % \mu_{t^{-1/\alpha} x}(A)
      \,,\quad x\in \Gamma\,, \; t>0\,.
\end{align*}
We have shown in the proof of Theorem~\ref{mainresult} that
$p_1^{\Gamma}(t^{-1/\alpha} x, y) / \P_{t^{-1/\alpha} x}
(\tau_{\Gamma} > 1)$ converges relatively uniformly to $n_1(y)$ as $t
\to \infty$, in the sense of the condition in Lemma~\ref{l:ru}, see \eqref{e:pGamma1}.
This yields the Yaglom limit $\mu(\rd x)=n_1(x)\rd x$.
\halmos

\subsection{Proof of Proposition~\ref{prop:qs}}\label{s:proofqs}
Recall that by (\ref{QSdensity}) we have $\mu(\rd z)= n_1(z)\;\rd z$.
By using   \eqref{e:ck} and 
(\ref{e:scaling}), 
for $A\subset\Rd$ we get
\begin{align*}
\mathbb{P}_\mu\left(\frac{X_t}{(t+1)^{1/\alpha}}\in A, \tau_\Gamma >t\right) =
& \int_{\Gamma} \int_{(t+1)^{1/\alpha}A} n_1(x) p^{\Gamma}_t(x,y)\rd y \rd x\\
  = & \int_{(t+1)^{1/\alpha}A} n_{t+1}(y) \rd y\\
  =& \int_{(t+1)^{1/\alpha}A} (t+1)^{-(d + \beta)/\alpha}
    n_1\left((t+1)^{-1/\alpha}y\right) \rd y\\
 = &(t+1)^{-\beta/\alpha}\int_{A} n_1(y)\rd y
 = (t+1)^{-\beta/\alpha}\mu(A)\,.
\end{align*}
In particular, $\mathbb{P}_\mu\left(\tau_\Gamma >t\right) =(t+1)^{-\beta/\alpha}$, which
ends
the proof.
\halmos

\section{Symmetric Cauchy process on half-line}\label{sec:cauchy}

  Let $d = \alpha = 1$ and $\Gamma = (0, \infty)$. Then $X_t$ is the
  symmetric Cauchy process on $\R$ and
  \[\tau_\Gamma=\inf\{t\geq 0:\, X_t <0\}\,,\]
which is sometimes called the ruin time.   
For this particular situation
we can add specific spectral information on the
Yaglom limit $\mu$.
Following \cite{MR2679702}, for $x > 0$, we let
  \begin{equation}\label{fr}
  r(x) = \frac{\sqrt{2}}{2 \pi} \int_0^{\infty} \frac{t}{(1 +
      t^2)^{5/4}} \exp \left( \frac{1}{\pi} \int_0^t \frac{\log s}{1 +
        s^2} \rd s \right) \e^{- t x} \rd t \end{equation}
  and
  \begin{equation}\label{fpsi} \psi(x) = \sin \left( x + \frac{\pi}{8} \right) - r(x)\,. \end{equation}

\begin{thm}\label{thmcauchy}
If $X_t$ is the
  symmetric Cauchy process on $\R$ and $\Gamma = (0, \infty)$, then $\mu$ has the density function
\begin{equation}
    \label{e:n1ycauchy}
    n_1(y) = \lim_{x \to 0+} \frac{p^{\Gamma}_1(x,
      y)}{\P_x(\tau_{\Gamma} > 1)} = \sqrt{\frac{\pi}{2}}
    \int_0^{\infty} \lambda^{1/2} \psi(\lambda y) \e^{- \lambda} \rd
    \lambda\,, \quad y>0\,.
  \end{equation}
\end{thm}
\begin{proof}
  By \cite[Example~3.2 and  3.4]{MR2075671}, we have
  \[ M(x) = \left( x \vee 0 \right)^{1/2}, \quad K(x) = \left( x
      \vee 0 \right)^{-1/2}\,, \quad x \in \R\,. \]

The Green function $G_{\Gamma}(x, y)$ is given by the well-known
explicit Riesz's formula:
\begin{equation}
  \label{e:GGamma1}
  G_{\Gamma}(x, y) = \frac{1}{\pi} \arcsin \sqrt{\frac{4 x y}{(x - y)^2}}\,,
  \quad x\,, \; y > 0\,;
\end{equation}
see \cite{MR0126885} or \cite[Theorem~3.3]{MR2506430} with $m
= 0$. Thus the constant $C_0$ defined in (\ref{eq:C0}) is
given by
\begin{equation}
  \label{e:C0}
  C_0 = \lim_{x \to 0 +} \frac{G_{\Gamma}(x, 1)}{M(x)} =
  \frac{2}{\pi}\,.
\end{equation}
For $t > 0$ we define
  \[ \xi(t) = \frac{1}{\pi} \frac{t}{(1+t^2)^{5/4}} \exp \left( -
      \frac{1}{\pi} \int_0^t \frac{\log s}{1 + s^2} \rd s \right)\,. \]
  Note that $\int_0^{\infty} \log s / (1 + s^2) \rd s = 0$. Thus
  \begin{equation}
    \label{e:ftinfty}
    \xi(t) \sim \frac{1}{\pi} t^{-3/2} \quad \text{ as } t \to \infty\,.
  \end{equation}
  It follows from \cite[Theorem~5]{MR2679702} that
  \begin{equation}
    \label{e:Pxtau1}
    \P_x(\tau_{\Gamma} \in \rd t) = \frac{1}{t} \xi \left( \frac{t}{x}
    \right) \rd t\,.
  \end{equation}
Therefore
  \begin{equation}
    \label{e:Pxtau0}
    \P_x \left( \tau_{\Gamma} > 1 \right) = \int_1^{\infty}
    \frac{1}{t} \xi \left( \frac{t}{x} \right) \rd t =
    \int_{x^{-1}}^{\infty} \frac{\xi(t)}{t} \rd t \sim \frac{2}{\pi}
    x^{1/2} = \frac{2}{\pi} M(x)
  \end{equation}
  as $x \to 0+$ and hence $C_1=C_0$, %for constant
  where $C_1$ is as defined in \eqref{Cjeden}.

  By \cite[(5.8)]{MR2679702} we have that $0 \leq r(x) \leq r(0) = \sin
  (\pi / 8)$ and
  $|\psi(x)| \leq 2$ for $x\geq 0$, where $r(x)$ and $\psi(x)$
  are as 
  defined in \eqref{fr} and \eqref{fpsi}, respectively.
  Further, by \cite[Theorem~2]{MR2679702} the function
  $\psi_{\lambda}(x) =\psi(\lambda x)$ is 
  %the 
  an eigenfunction of the semigroup
  $P^{\Gamma}_t$ acting on $C(\Gamma)$, with the
  % corresponding to 
  eigenvalue
  $\e^{-\lambda t}$. Thus
  \begin{equation}
    \label{e:p1txy}
    p^{\Gamma}_t(x, y) = \frac{2}{\pi} \int_0^{\infty}
    \psi_{\lambda}(x) \psi_{\lambda}(y) \e^{-\lambda t} \rd \lambda\,;
  \end{equation}
  see \cite[(7.4)]{MR2679702}. Note that
  \[ r'(x) = - \frac{\sqrt{2}}{2 \pi} \int_0^{\infty} \frac{t^2}{(1 +
      t^2)^{5/4}} \exp \left( \frac{1}{\pi} \int_0^t \frac{\log s}{1 +
        s^2} \rd s \right) \e^{- t x} \rd t\,. \]
  Since
  \[ - \int_0^t \frac{\log s}{1 + s^2} \rd s = \int_t^{\infty}
    \frac{\log s}{1+s^2} \rd s \]
  is positive for all $t > 0$ and it is regularly varying at $\infty$ of index
  $-1$, the following estimates hold  for $x>0$:
  \[ \int_0^1 \frac{t^2}{(1 + t^2)^{5/4}} \exp \left( \frac{1}{\pi}
      \int_0^t \frac{\log s}{1+s^2} \rd s \right) \e^{-t x} \rd t \leq
    1\,, \]
  \begin{align*}
    & \left| \int_1^{\infty} \frac{t^2}{(1 + t^2)^{5/4}} \exp \left(
      \frac{1}{\pi} \int_0^t \frac{\log s}{1+s^2} \rd s \right) \e^{-t
      x} \rd t - \int_1^{\infty} t^{-1/2} \e^{-t x} \rd t \right| \\
    \leq & \int_1^{\infty} \left| \frac{t^2}{(1 + t^2)^{5/4}} -
           t^{-1/2} \right|  \rd t +
           \int_1^{\infty} t^{-1/2} \left| 1 - \exp \left(
           \frac{1}{\pi} \int_0^t \frac{\log s}{1+s^2} \rd s \right)
           \right| \rd t \\
    \leq & 1 + \frac{1}{\pi} \int_1^{\infty} t^{-1/2} \int_t^{\infty}
           \frac{\log s}{1 + s^2} \rd s \rd t
    < \infty
  \end{align*}
  and
  \[ \left| \int_1^{\infty} t^{-1/2} \e^{-t x} \rd t - \sqrt{\pi}
      x^{-1/2} \right| \leq \int_0^1 t^{-1/2} \e^{-t x} \rd t \leq 2\,. \]
  Thus there exists a constant $c > 0$ such that for $x > 0$,
  \begin{equation}
    \label{e:rxr0}
    \left| r(x) - r(0) - \sqrt{\frac{2}{\pi}} x^{1/2} \right| \leq
    \int_0^x \left| r'(s) + \frac{1}{\sqrt{2 \pi}} s^{-1/2} \right|
    \rd s \leq c x
  \end{equation}
  and
  \begin{equation}
    \label{e:psix0}
    \left| \psi(x) - \sqrt{\frac{2}{\pi}} x^{1/2} \right| \leq \left|
      \sin \left( x + \frac{\pi}{8} \right) - \sin \frac{\pi}{8}
    \right| + \left| r(x) - r(0) - \sqrt{\frac{2}{\pi}} x^{1/2}
    \right| \leq c x\,.
  \end{equation}
  The above inequalities and \eqref{e:p1xy} imply that
  \begin{equation}
    \label{e:p1Gammaa}
    \left| p^{\Gamma}_1(x, y) - \sqrt{\frac{2}{\pi}} x^{1/2}
      \int_0^{\infty} \lambda^{1/2} \psi(\lambda y) \e^{-\lambda} \rd
      \lambda \right| \leq c x \int_0^{\infty} \lambda \psi(\lambda
    y) \e^{- \lambda} \rd \lambda\,.
  \end{equation}
The identity (\ref{e:n1ycauchy}) now follows from \eqref{e:Pxtau0} and \eqref{e:p1Gammaa}.
Since we have \eqref{QSdensity}, the proof is complete.
  % of Theorem \ref{thmcauchy}.
  \end{proof}
{\bf Acknowledgements.} We thank Victor Rivero for discussions on quasi-stationary distributions. We thank Gavin Armstrong for  comments on the paper.

\end{document}